\documentclass{amsart}
\usepackage{amssymb}
\usepackage{amsmath}
\usepackage{amsthm}
\newtheorem{theorem}{Theorem}[section]
\newtheorem{lemma}[theorem]{Lemma}
\newtheorem{proposition}[theorem]{Proposition}

\newtheorem{definition}[theorem]{Definition\rm}
\newtheorem{remark}{\it Remark\/}

\newcommand{\diam}{\mathrm{diam}\,}

\newcommand{\vvv}{|\!|\!|}
\newcommand{\CC}{\mathbb{C}}

\newcommand{\Gr}{\mathrm{Gr}}

\newcommand{\Mat}{\mathrm{Mat}}

\begin{document}
\title[Lower bound for the joint spectral radius]{A rapidly-converging lower bound for the joint spectral radius via multiplicative ergodic theory}\author{Ian D. Morris}
\begin{abstract}
We use ergodic theory to prove a quantitative version of a theorem of M. A. Berger and Y. Wang, which relates the joint spectral radius of a set of matrices to the spectral radii of finite products of those matrices. The proof rests on a theorem asserting the existence of a continuous invariant splitting for certain matrix cocycles defined over a minimal homeomorphism and having the property that all forward products are uniformly bounded. MSC primary 15A18, 37H15, 65F15, secondary 37M25.\end{abstract}
\maketitle
\section{Introduction}
Let $\mathsf{A}$ be a bounded set of $d \times d$ complex matrices. The joint spectral radius of $\mathsf{A}$, introduced by G.-C. Rota and G. Strang in \cite{RS}, is defined to be the quantity
\begin{equation}\label{oryx}
\varrho(\mathsf{A}):= \lim_{n \to \infty} \sup\left\{\|A_n\cdots A_1\|^{1/n}\colon A_i \in \mathsf{A}\right\},\end{equation}
where $\|\cdot\|$ denotes any norm on $\mathbb{C}^d$. This is easily seen to yield a finite value which is well-defined with respect to the choice of norm. The joint spectral radius arises naturally in a range of topics including control and stability \cite{Ba,Gu,Koz}, coding theory \cite{MO}, wavelet regularity \cite{DL,DL0,Mae2}, numerical solutions to ordinary differential equations \cite{GZ}, and combinatorics \cite{DST}. The problem of computing the joint spectral radius of a finite set of matrices has therefore attracted substantial research interest \cite{BN,Gr,GWZ,Koz,LW,Mae,Parr,BT1,Wirth}. In this article we shall prove a new estimate relevant to the computation of the joint spectral radius.

Let $\Mat_d(\mathbb{C})$ denote the set of all $d \times d$ complex matrices. The following theorem was proved by M. A. Berger and Y. Wang \cite{BW}, having originally been conjectured by I. Daubechies and J. C. Lagarias \cite{DL}:
\begin{theorem}[Berger-Wang formula]\label{BWF}
Let $\mathsf{A} \subset \Mat_d(\mathbb{C})$ be bounded. Then
\begin{equation}\label{groke}\varrho(\mathsf{A}) = \limsup_{n \to \infty} \sup\left\{\rho(A_n \cdots A_1)^{1/n}\colon A_i \in \mathsf{A}\right\},\end{equation}
where $\rho(A)$ denotes the ordinary spectral radius of a matrix $A$.
\end{theorem}
Some alternative proofs are given in \cite{Bochi,E,SWP}. In this article we shall study the rate of convergence in the expression \eqref{groke}. This has potential implications for some approaches to the computation of the joint spectral radius such as the algorithm given by G. Gripenberg \cite{Gr}.

Let $\|\cdot\|$ be any norm on $\mathbb{C}^d$. For each $n \in \mathbb{N}$ define
\[\varrho_n^+(\mathsf{A},\|\cdot\|)= \sup\left\{\|A_n\cdots A_1\|^{1/n} \colon A_i \in \mathsf{A}\right\},\]
\[\varrho_n^-(\mathsf{A})= \sup\left\{\rho(A_n\cdots A_1)^{1/n} \colon A_i \in \mathsf{A}\right\}.\]
For fixed $\mathsf{A}$ it is clear that $\varrho_{n+m}^+(\mathsf{A},\|\cdot\|) \leq \varrho^+_n(\mathsf{A},\|\cdot\|) \varrho_m^+(\mathsf{A},\|\cdot\|)$ for all $n,m \in \mathbb{N}$, which implies that the limit in \eqref{oryx} may be replaced by an infimum. Conversely, since $\rho(A^m)^{1/m}=\rho(A)$ for all $m \in \mathbb{N}$ and any matrix $A$, one may easily show that  $\varrho_{nm}^-(\mathsf{A}) \geq \varrho_n^-(\mathsf{A})$ for every $n,m \in \mathbb{N}$ and hence  the limit superior in \eqref{groke} is also a supremum. In general this limit superior can fail to be a limit, a simple example being
\[\mathsf{A} = \left\{\left(\begin{array}{cc}0&2\\\frac{1}{2}&0\end{array}\right),\left(\begin{array}{cc}0&1\\1&0\end{array}\right)\right\}.\]
In this article we shall present a proof of the following theorem, which extends Theorem \ref{BWF} in the case where $\mathsf{A}$ is finite:
\begin{theorem}\label{QBWF}
Let $\mathsf{A}$ be a finite set of $d \times d$ complex matrices. Then for every integer $r \in \mathbb{N}$,
 \[\left|\varrho(\mathsf{A})- \max_{1 \leq k \leq n} \varrho_k^-(\mathsf{A}) \right| = O\left(\frac{1}{n^r}\right).\]
\end{theorem}
Theorem \ref{QBWF} implies in particular that if we wish to compute $\varrho(\mathsf{A})$ to within accuracy $\varepsilon$ by means of brute-force estimation of the values $\varrho^-_n(\mathsf{A})$, then the number of matrix products which must be evaluated increases at a slower-than-stretched-exponential rate as a function of $1/\varepsilon$. However, it should be noted that the arguments used in this paper do not seem to be well-suited to the production of an \emph{effective} estimate for the quantity $\varrho(\mathsf{A})$.

Two estimates related to Theorem \ref{QBWF} have been established previously. By a theorem of J. Bochi \cite{Bochi}, there exist for each $d \in \mathbb{N}$ a constant $C_d>0$ and an integer $m \in \mathbb{N}$ such that $\varrho(\mathsf{A}) \leq C_d\max_{1 \leq k \leq m} \varrho^-_k(\mathsf{A})$ for every bounded set $\mathsf{A} \subset \Mat_d(\mathbb{C})$. An easy consequence is the estimate
\[\left|\varrho(\mathsf{A})- \max_{1 \leq k \leq mn} \varrho_{k}^-(\mathsf{A}) \right| \leq \left(1-C^{-1/n}_d\right)\varrho(\mathsf{A}) = O\left(\frac{1}{n}\right).\]
In the other direction, F. Wirth \cite{Wirth} gives the general bound
\[\left|\varrho(\mathsf{A})-\varrho_n^+(\mathsf{A},\|\cdot\|)\right| = O\left(\frac{\log n}{n}\right)\]
for any norm $\|\cdot\|$ on $\mathbb{C}^d$ and bounded set $\mathsf{A}\subset \Mat_d(\mathbb{C})$. This estimate improves to $O(1/n)$ if it is assumed that there does not exist a linear space $V$ such that $\{0\} \subset V \subset \mathbb{C}^d$ and $AV \subseteq V$ for every $A \in \mathsf{A}$. Unlike Bochi's estimate, the constant in Wirth's estimate may vary between sets of matrices $\mathsf{A}$. The example
\[\mathsf{A} = \left\{\left(\begin{array}{cc}2&2\\0&0\end{array}\right),\left(\begin{array}{cc}1&1\\1&1\end{array}\right)\right\}\]
shows that Wirth's estimate cannot be improved directly: taking $\|\cdot\|$ to be the Euclidean norm we obtain $\varrho_n^+(\mathsf{A},\|\cdot\|) = 2^{1+1/2n}$ for each $n \in \mathbb{N}$, whereas $\varrho^-_1(\mathsf{A})=2$ and hence $\varrho(\mathsf{A})=2$.

The proof of Theorem \ref{QBWF} has some points of resemblance to the proof of Theorem \ref{BWF} given by L. Elsner \cite{E}, which we now elaborate upon. Elsner's proof runs essentially as follows. If $\varrho(\mathsf{A})=0$ then the result is trivially true. Otherwise, by normalising we may take $\varrho(\mathsf{A})=1$. We then reduce to the case where a uniform bound exists for products of elements of $\mathsf{A}$, and hence there exists a compact subset of $\Mat_d(\mathbb{C})$ which contains $\{A_{n}\ldots A_{1} \colon A_i \in \mathsf{A}\}$ for every $n$. By using the pigeonhole principle on open $\varepsilon$-balls in $\Mat_d(\mathbb{C})$ and in $\mathbb{C}^d$, we can then guarantee the existence of a finite sequence $A_1,\ldots,A_n$ and a vector $v$ belonging to the unit sphere of $\mathbb{C}^d$ such that $A_n\cdots A_1v$  is close to $v$ and therefore the spectral radius of $A_{n}\cdots A_{1}$ is close to 1.

In our proof of Theorem \ref{QBWF} we make this strategy quantitative, replacing the pigeonhole principle with a more delicate recurrence argument. In order to achieve this we first prove a theorem describing the dynamical structure of matrix sequences $(A_i)$ with the property that $\|A_{n}\cdots A_{1}\|$ is large for all $n$, and additionally we achieve some understanding of the structure of the orbits in $\mathbb{C}^d$ which are induced by the action of such sequences. The bulk of this paper, therefore, is concerned with proving a theorem on the dynamical structure of these `extremal' sequences. We describe these ideas in detail in the following section.
\section{Linear cocycles}

At this point it is convenient to establish some notation and definitions. In the remainder of this article the symbol $\|\cdot\|$ shall be used to denote the Euclidean norm on $\mathbb{C}^d$, whereas the symbol $\vvv\cdot\vvv$ shall be used to denote an \emph{extremal norm} on $\mathbb{C}^d$, which will be defined shortly. In either case we shall also use the symbols $\|\cdot\|$ and $\vvv \cdot\vvv$ to denote the corresponding operator norms induced on $\Mat_d(\mathbb{C})$. Throughout this article we adhere to the convention  $\log 0 := -\infty$. 

Let $T \colon X \to X$ be a continuous transformation of a compact metric space. A \emph{cocycle} over $T$ with values in the complex matrices is a function $\mathcal{A} \colon X \times \mathbb{N} \to \Mat_d(\mathbb{C})$ such that for each $x \in X$ and $n,m\in \mathbb{N}$
\[\mathcal{A}(x,n+m) = \mathcal{A}(T^nx,m)\mathcal{A}(x,m).\]
We say that the cocycle $\mathcal{A}$ is continuous if $\mathcal{A}(\cdot,n)$ is a continuous function from $X$ to $\Mat_d(\mathbb{C})$ for each $n \in \mathbb{N}$. Abusing notation somewhat, we shall sometimes denote $\mathcal{A}(x,1)$ simply by $\mathcal{A}(x)$. Since for each $x, n$
\[\mathcal{A}(x,n) = \mathcal{A}(T^{n-1}x)\cdots\mathcal{A}(Tx)\mathcal{A}(x)\]
the cocycle $\mathcal{A}\colon X \times \mathbb{N} \to \Mat_d(\mathbb{C})$ is completely determined by the function $\mathcal{A} \colon X \to \Mat_d(\mathbb{C})$. Whilst it will always be the case in this article that the map $T$ is a homeomorpism, we do not assume that the values of the function $\mathcal{A}$ are invertible matrices, and so we cannot in general extend $\mathcal{A}$ to an invertible cocycle defined on $X \times \mathbb{Z}$.
 
For $0 \leq p \leq d$ we let $\Gr(p,d)$ denote the set of all $p$-dimensional subspaces of $\mathbb{C}^d$. This set may be identified with the set of all orthogonal projections from $\mathbb{C}^d$ onto a $p$-dimensional subspace. We equip $\Gr(p,d)$ with the standard metric given by
\[d_{\Gr}(V,W) := \|P^{\perp}_V - P^{\perp}_W\|\]
where $P^{\perp}_Z$ denotes the linear map given by orthogonal projection onto $Z$. This metric makes $\Gr(p,d)$ a compact metric space.
We shall say that a function $\mathcal{V} \colon X \to \Gr(p,d)$ is \emph{forward-invariant} under a cocycle $\mathcal{A}$ if $\mathcal{A}(x,n)\mathcal{V}(x) \subseteq \mathcal{V}(T^nx)$ for all $x \in X$ and $n \in\mathbb{N}$.

We begin by establishing the following general theorem which will later be applied to study matrix cocycles associated to a compact set $\mathsf{A} \subset \Mat_d(\mathbb{C})$.
\begin{theorem}\label{onepointone}
Let $T \colon X \to X$ be a minimal homeomorphism of a compact metric space, and let $\mathcal{A} \colon X \times \mathbb{N} \to \Mat_d(\CC)$ be a continuous linear cocycle. Suppose that there exists $M>0$ such that $\|\mathcal{A}(x,n)\| \leq M$ for all $x \in X$ and all $n \in \mathbb{N}$. 
Then there exist an integer $0 \leq p \leq d$ and continuous forward-invariant functions $\mathcal{V} \colon X \to \Gr(p,d)$, $\mathcal{W} \colon X \to \Gr(d-p,d)$ such that $\mathcal{V}(x) \oplus \mathcal{W}(x) = \mathbb{C}^d$ for all $x \in X$. Moreover there exist constants $C,\delta>0$ and $\xi \in (0,1)$ such that for all $x \in X$ and $n \in \mathbb{N}$, $\|\mathcal{A}(x,n)v\| \geq \delta$ for every $v \in \mathcal{V}(x)$ and $\|\mathcal{A}(x,n)w\| \leq C\xi^n\|w\|$ for every $w \in \mathcal{W}(x)$.

The moduli of continuity of $\mathcal{V}$ and $\mathcal{W}$ admit the following description. If $n \in \mathbb{N}$ is given, suppose that $x,y \in X$ satisfy
\[\max\left\{ \|\mathcal{A}(x,2n) - \mathcal{A}(y,2n)\|,\|\mathcal{A}(x,n)-\mathcal{A}(y,n)\|\right\}\leq\delta\xi^n.\]
Then $d_{\Gr}(\mathcal{V}(x),\mathcal{V}(y)) \leq \tilde C \xi^n$ for some constant $\tilde C>0$. Similarly, if $x,y \in X$ satisfy
\[\|\mathcal{A}(x,n)-\mathcal{A}(y,n)\| \leq \xi^n\]
then $d_{\Gr}(\mathcal{W}(x),\mathcal{W}(y)) \leq \tilde C \xi^n$.

For each $x \in X$ let $P(x)$ denote the projection with image $\mathcal{V}(x)$ and kernel $\mathcal{W}(x)$. Then $P(x)$ depends continuously on $x$, and in particular there exists $K>0$ such that
\[\|P(x)-P(y)\| \leq K\big[d_{\Gr}(\mathcal{V}(x),\mathcal{V}(y))+d_{\Gr}(\mathcal{W}(x),\mathcal{W}(y))\big]\]
for all $x,y \in X$.
\end{theorem}
While Theorem \ref{onepointone} has a number of features in common with the classical multiplicative ergodic theorem of V. I. Oseledec (see e.g. \cite{Krengel}) our proof is direct and does not make use of any prior multiplicative ergodic theorems. Indeed, since in general we wish to work with non-invertible matrices, the standard statement of Oseledec's theorem does not give the existence even of a measurable splitting of the type given above, giving only an invariant flag (though see \cite{FLQ}). The proof of Theorem \ref{onepointone} does however incorporate ideas used in the proofs of Oseledec's theorem given by M. S. Raghunathan \cite{Rag} and D. Ruelle \cite{Ruelle}.

Note that if $p=0$ then the conclusions of the theorem are somewhat vacuous, and in applications further analysis is needed to show that this situation does not arise.

In order to apply this theorem in the desired context we require some further definitions. We shall say that $\mathsf{A} \subset \Mat_d(\mathbb{C})$ is \emph{product bounded} if there exists $M>0$ such that for every $n \in \mathbb{N}$ we have $\|A_{n}\cdots A_{1}\| \leq M$ for every finite sequence $(A_{n},\ldots,A_{1})\in \mathsf{A}^n$. Note if such a uniform bound holds for $\mathsf{A}$ with respect to some norm on $\mathbb{C}^d$ then it holds with respect to all such norms, subject to variation in the constant $M$. We shall say that a norm $\vvv\cdot\vvv$ on $\mathbb{C}^d$ is an \emph{extremal norm} for $\mathsf{A}$ if $\vvv A\vvv \leq \varrho(\mathsf{A})$ for all $A \in \mathsf{A}$. If $\varrho(\mathsf{A})>0$ then an extremal norm exists for $\mathsf{A}$ if and only if  $\varrho(\mathsf{A})^{-1}\mathsf{A}$ is product bounded \cite{Koz,RS}.

Given a compact set $\mathsf{A} \subset \Mat_d(\mathbb{C})$, let us define a metric on $\mathsf{A}^{\mathbb{Z}}$ by
\[d\left[(A_i)_{i \in \mathbb{Z}},(B_i)_{i \in \mathbb{Z}}\right]:= \sum_{i \in \mathbb{Z}} \frac{\|A_i-B_i\|}{2^{|i|}}.\]
If $\mathsf{A}$ is compact then $(\mathsf{A}^{\mathbb{Z}},d)$ is compact. We define the \emph{shift map} $T \colon \mathsf{A}^{\mathbb{Z}} \to \mathsf{A}^{\mathbb{Z}}$ by $T[(A_i)_{i \in \mathbb{Z}}]= (A_{i+1})_{i \in \mathbb{Z}}$. The shift map is a Lipschitz homeomorphism of $\mathsf{A}^{\mathbb{Z}}$. Let $\mathcal{A} \colon \mathsf{A}^{\mathbb{Z}} \to \Mat_d(\mathbb{C})$ be given by $\mathcal{A}[(A_i)_{i \in \mathbb{Z}}]=A_1$, and let $\mathcal{A}(x,n) = \mathcal{A}(T^{n-1}x)\cdots \mathcal{A}(x)$ for all $(x,n) \in \mathsf{A}^{\mathbb{Z}} \times \mathbb{N}$ so that $\mathcal{A} \colon \mathsf{A}^{\mathbb{Z}} \times \mathbb{N} \to \Mat_d(\mathbb{C})$ is a continuous cocycle. For each $n \in \mathbb{N}$ we have
\[\varrho^+_n(\mathsf{A},\|\cdot\|) = \sup\left\{\|\mathcal{A}(x,n)\|^{1/n} \colon x \in \mathsf{A}\right\}\]
and
\[\varrho^-_n(\mathsf{A}) = \sup\left\{\rho(\mathcal{A}(x,n))^{1/n} \colon x \in \mathsf{A}\right\}.\]
As a consequence we deduce
\[\log \varrho(\mathsf{A}) = \lim_{n \to \infty} \sup_{x \in \mathsf{A}^{\mathbb{Z}}} \frac{1}{n}\log \|\mathcal{A}(x,n)\|,\]
a formulation which is particularly amenable to study using ergodic theory via Theorem \ref{SUSAET} below.

Combining Theorem \ref{onepointone} with some supplementary results given in section 3 below, we obtain the following:
 \begin{theorem}\label{Tech}
Let $\mathsf{A} \subset \Mat_d(\mathbb{C})$ be a compact set such that $\varrho(\mathsf{A})=1$, and suppose that $\mathsf{A}$ is product bounded. Let $\vvv \cdot \vvv$ be any extremal norm for $\mathsf{A}$ and define
\[Y := \left\{ x\in \mathsf{A}^{\mathbb{Z}} \colon \vvv \mathcal{A}(x,n)\vvv =1\text{ }\forall\text{ }n \in \mathbb{N}\right\}.\]
Then the set $Y$ is a compact, nonempty subset of $\mathsf{A}^{\mathbb{Z}}$ such that $TY \subseteq Y$.

Let $Z\subseteq Y$ be any invariant subset such that $T \colon Z \to Z$ is minimal. Then there exists an integer $1\leq p \leq d$ such that the following properties hold. There exist H\"older continuous invariant functions $\mathcal{V} \colon Z \to \Gr(p,d)$, $\mathcal{W} \colon Z \to \Gr(d-p,d)$ such that $\mathcal{V}(x)\oplus \mathcal{W}(x) =\mathbb{C}^d$ for each $x\in Z$. There exist constants $C>0$, $\xi \in (0,1)$ such that for all $x \in Z$ and $n \in \mathbb{N}$, $\vvv \mathcal{A}(x,n)v \vvv = \vvv v \vvv$ for all $v \in \mathcal{V}(x)$ and $\vvv \mathcal{A}(x,n) w\vvv \leq C\xi^n \vvv w \vvv$ for all $w \in \mathcal{W}(x)$. If for each $x \in Z$ we let $P(x)$ denote the projection with image $\mathcal{V}(x)$ and kernel $\mathcal{W}(x)$ then $P \colon Z \to \Mat_d(\mathbb{C})$ is H\"older continuous.
\end{theorem}
To obtain Theorem \ref{QBWF} we combine this result with an estimate due to X. Bressaud and A. Quas on the approximation via periodic orbits of closed invariant subsets of shift transformations over finite alphabets (cf. \cite{BQ}).

The remainder of this article is structured as follows. In section \ref{three} we establish some results in subadditive ergodic theory which are needed in the proofs of Theorems \ref{onepointone} and \ref{Tech}. In sections \ref{four} and \ref{five} we prove these two theorems, and in section \ref{six} we give the proof of Theorem \ref{QBWF}. Finally, in section \ref{seven} we describe the obstructions to improving the error term in Theorem \ref{QBWF} and to extending that theorem to the case of infinite compact sets $\mathsf{A}$.

\section{Subadditive ergodic optimisation}\label{three}

The recently-developed topic of \emph{ergodic optimisation} is concerned with the following problem. Given a continuous dynamical system $T \colon X \to X$ defined on a compact metric space, and some continuous (or only upper semi-continuous) function $f \colon X \to \mathbb{R}$, one studies the greatest possible linear growth rate of the sequence $\sum_{j=0}^{n-1}f(T^jx)$ as $x$ varies over $X$, which is equal to the supremum of all possible values of the integral of $f$ with respect to a $T$-invariant probability measure on $X$. Problems which are considered include the identification and approximation of those invariant measures which attain this supremum. Some recent research articles in this area include \cite{B1,BJ,Bremont,BQ,CLT,CG,J,YH}. In this section we generalise some (mostly standard) results from ergodic optimisation to the context of subadditive ergodic theory, with the aim of applying these results to the proof of Theorem \ref{onepointone}. For parallels of these results in the additive case we direct the reader to \cite{J}.
 
Throughout this section we assume that $X$ is a compact metric space and $T \colon X \to X$ a continuous transformation. We let $\mathcal{M}$ denote the set of all Borel probability measures on $X$ and let $\mathcal{M}_T$ denote the subset consisting of all $T$-invariant Borel probability measures. We equip $\mathcal{M}$ and $\mathcal{M}_T$ with the weak-* topology, under which both sets are compact and metrisable \cite{W}. 

We in fact only require the results established below in the case where $f \colon X \to \mathbb{R} \cup \{-\infty\}$ is continuous, but the case in which $f$ is only taken to be upper semi-continuous is included also since this does not require any modification to the proofs. The following simple result is important enough to be worth stating explicitly:
\begin{lemma}\label{A1}
Let $f \colon X \to \mathbb{R} \cup \{-\infty\}$ be upper semi-continuous. Then the map from $\mathcal{M}$ to $\mathbb{R} \cup \{-\infty\}$ defined by $\mu \mapsto \int f\,d\mu$ is upper semi-continuous.
\end{lemma}
\begin{proof}
Recall that a function from a metrisable space to $\mathbb{R} \cup \{-\infty\}$ is upper semi-continuous if and only if it is equal to the pointwise limit of a decreasing sequence of continuous functions taking values in $\mathbb{R}$ (see e.g. \cite[ch. IX]{Bourbaki}). Let $(f_i)_{i=1}^\infty$ be such a sequence converging pointwise to $f$. For each $i$ the map $\mu \mapsto \int f_i\,d\mu$ is clearly real-valued and is by definition weak-* continuous, and for each $\mu$ the sequence $(\int f_i\,d\mu)_{i=1}^\infty$ decreases to $\int f\,d\mu$ by the Monotone Convergence Theorem.\end{proof}
Recall that a sequence $(a_n)_{n=1}^\infty$ such that $a_n \in \mathbb{R}\cup\{-\infty\}$ for each $n$ is said to be \emph{subadditive} if $a_{n+m} \leq a_n + a_m$ for all $n, m \in \mathbb{N}$. If this is the case then
\[\lim_{n \to \infty}\frac{a_n}{n} = \inf_{n \geq 1} \frac{a_n}{n} \in \mathbb{R}\cup\{-\infty\}.\]
\begin{definition}
We say that a sequence $(f_n)_{n=1}^\infty$ of functions from $X$ to $\mathbb{R} \cup \{-\infty\}$ is \emph{subadditive} if $f_{n+m}(x) \leq f_n(T^mx)+f_m(x)$ for all $n,m \in \mathbb{N}$ and all $x \in X$. 
\end{definition}
If $\mu \in \mathcal{M}_T$ and $(f_n)_{n=1}^\infty$ is a subadditive sequence of upper semi-continuous functions then the sequence $(\int f_n\,d\mu)_{n=1}^\infty$ is easily seen to be subadditive. If in addition $\mu \in \mathcal{M}_T$ is ergodic, then the Subadditive Ergodic Theorem asserts that for $\mu$-a.e. $x \in X$
\[\lim_{n \to \infty} \frac{1}{n}f_n(x) = \lim_{n \to \infty} \frac{1}{n} \int f_n\,d\mu = \inf_{n \geq 1} \frac{1}{n}\int f_n\,d\mu,\]
see e.g. \cite{Krengel}. This motivates the following definition.
 \begin{definition}
Let $(f_n)$ be a subadditive sequence of upper semi-continuous functions from $X$ to $\mathbb{R} \cup \{-\infty\}$. The \emph{maximum ergodic average} of $(f_n)$ is defined to be the quantity
 \[\beta[(f_n)]:=\sup_{\mu \in \mathcal{M}_T} \lim_{n \to \infty} \frac{1}{n} \int f_n\,d\mu = \sup_{\mu \in \mathcal{M}_T} \inf_{n \geq 1} \frac{1}{n} \int f_n\,d\mu.\]
 We define $\mathcal{M}_{\max}[(f_n)]$ to be the set of all $\mu \in \mathcal{M}_T$ for which this supremum is attained.
 \end{definition}
 The following important result, called the \emph{semi-uniform subadditive ergodic theorem} in \cite{StSt}, is due independently to S. J. Schreiber \cite{Sch} and to R. Sturman and J. Stark \cite{StSt}. Since the version which we use is somewhat more general than those given by Schreiber and Sturman-Stark, we include a proof in the appendix. 
 \begin{theorem}[Semi-uniform subadditive ergodic theorem]\label{SUSAET}
 Let $(f_n)$ be a subadditive sequence of upper semicontinuous functions from $X$ to $\mathbb{R} \cup \{-\infty\}$. Then
 \[\beta[(f_n)]=\lim_{n \to \infty}\sup_{x \in X}\frac{1}{n}f_n(x) = \sup_{x \in X} \limsup_{n \to \infty} \frac{1}{n} f_n(x)= \lim_{n \to \infty} \sup_{\mu \in \mathcal{M}_T} \frac{1}{n}\int f_n\,d\mu.\]
 \end{theorem}
We next prove some results describing the structure of the set $\mathcal{M}_{\max}[(f_n)]$ for a subadditive sequence $(f_n)$.
 \begin{lemma}\label{nonemptiness}
 Let $(f_n)$ be a subadditive sequence of upper semi-continuous functions from $X$ to $\mathbb{R} \cup \{-\infty\}$. Then $\mathcal{M}_{\max}[(f_n)]$ is a compact subset of $\mathcal{M}_T$ and contains an ergodic measure.
 \end{lemma}
 \begin{proof}
If $\beta[(f_n)]=-\infty$ then $\mathcal{M}_{\max}[(f_n)]=\mathcal{M}_T$ and the result is trivial. We therefore assume $\beta[(f_n)] \in \mathbb{R}$. By Lemma \ref{A1} each of the maps $\mu \mapsto (1/n)\int f_n\,d\mu$ is upper semi-continuous, and it follows from this that the map $\mu \mapsto \inf_{n \geq 1}(1/n)\int f_n\,d\mu$ is upper semi-continuous also. Since $\mathcal{M}_T$ is compact this implies that $\mathcal{M}_{\max}[(f_n)]$ is compact and nonempty.
 
Let $\mu \in \mathcal{M}_{\max}[(f_n)]$. By the ergodic decomposition theorem, there exist a measurable space $(\Omega,\mathcal{F},\mathbb{P})$ and measurable function $\mu_{(\cdot)} \colon \Omega \to \mathcal{M}_T$ such that $\mu_\omega$ is ergodic $\mathbb{P}$-a.e. and such that for each Borel set $A \subseteq X$ the map $\omega \mapsto \mu_\omega(A)$ is $\mathcal{F}$-measurable  and satisfies $\mu(A) = \int_\Omega \mu_\omega(A) d\mathbb{P}(\omega)$. For each $r,k \in \mathbb{N}$ define
 \[\mathcal{Z}_{r,k} = \left\{\omega \in \Omega \colon  \frac{1}{r}\int f_r\,d\mu_\omega<\beta[(f_n)]-\frac{1}{k}\right\} \in \mathcal{F}.\]
If one has $\mathbb{P}(\mathcal{Z}_{r,k})>0$ for some $r,k \in \mathbb{N}$ then
 \[\beta[(f_n)] \leq \frac{1}{r}\int f_r\,d\mu = \frac{1}{r}\int_\Omega \int f_r\,d\mu_\omega d\mathbb{P}(\omega) \leq \beta[(f_n)]\left(1-\frac{\mathbb{P}(\mathcal{Z}_{r,k})}{k}\right)<\beta[(f_n)],\]
a contradiction. We conclude that $\mathbb{P}(\mathcal{Z}_{r,k})=0$ for all $r,k \in \mathbb{N}$ and thus
 \[\mathbb{P}\left(\left\{\omega \in \Omega \colon \mu_\omega \in \mathcal{M}_{\max}[(f_n)]\right\}\right)= \mathbb{P}\left(\left\{ \omega \in \Omega \colon \inf_{r \geq 1} \frac{1}{r} \int f_r\,d\mu_\omega \geq \beta[(f_n)] \right\}\right)=1.\]
 In particular there exists $\omega \in \Omega$ such that $\mu_\omega$ is ergodic and $\mu_\omega \in \mathcal{M}_{\max}[(f_n)]$.
 \end{proof}
 The following result gives an analogue of the \emph{subordination principle} described by T. Bousch \cite{B1}. While only parts of its statement are actually required in this article, the full statement is included for the sake of interest.
 \begin{lemma}\label{subordinationsprinzip}
 Let $(f_n)$ be a subadditive sequence of upper semi-continuous functions from $X$ to $\mathbb{R} \cup \{-\infty\}$, and suppose that there exists $\lambda \in \mathbb{R}$ such that $\sup \{ f_n(x) \colon x \in X\} =n\lambda$ for all $n \in \mathbb{N}$. Then $\beta[(f_n)] = \lambda$ and if we define for each $n$
\[Y_n := \left\{x \in X \colon f_n(x) =n\lambda\right\}\]
then $Y:= \bigcap_{n=1}^\infty Y_n$ is compact and nonempty and satisfies $TY \subseteq Y$. Furthermore, each $\mu \in \mathcal{M}_T$ satisfies $\mathcal{M}_{\max}[(f_n)]$ if and only if it satisfies $\mu(Y)=1$.
 \end{lemma}
 \begin{proof}
Since $\sup f_n = n\lambda$ for each $n$ it is clear that each $Y_n$ is closed and that $\beta[(f_n)] \leq \lambda$. If $x \in Y_{n+1}$ then since
\[(n+1)\lambda = f_{n+1}(x) \leq f_n(x) + f_1(T^nx) \leq f_n(x) + \lambda\leq (n+1)\lambda\]
we have $x \in Y_n$ also. It follows that the intersection $\bigcap_{n=1}^\infty Y_n$ is nonempty. If $x \in Y$ then for each $n \in \mathbb{N}$ we have
\[(n+1)\lambda = f_{n+1}(x) \leq f_n(Tx) + f_1(x) = f_n(Tx) +\lambda \leq (n+1)\lambda\]
so that $f_n(Tx)=n\lambda$, and we deduce that $Tx \in Y$. By the Krylov-Bogolioubov theorem there exists at least one invariant measure $\mu$ such that $\mu(Y)=1$. Since then $n^{-1}\int f_n\,d\mu = \lambda$ for every $n \in \mathbb{N}$ it follows that $\beta[(f_n)] \geq \lambda$, and this argument also shows that if $\nu(Y)=1$ and $\nu \in \mathcal{M}_T$ then necessarily $\nu \in \mathcal{M}_{\max}[(f_n)]$. Finally, suppose that $\mu \in \mathcal{M}_T$ with $\mu(X \setminus Y)>0$. Choose $r \in \mathbb{N}$, $\delta>0$ and a nonempty open set $U \subseteq X \setminus Y$ such that $\mu(U)>0$ and $f_r(x)<r(\lambda-\delta)$ for all $x\in U$. We have
\[\inf_{n \geq 1}\frac{1}{n} \int f_n\,d\mu \leq \frac{1}{r}\int f_r\,d\mu \leq (1-\mu(U))\lambda +\mu(U)(\lambda-\delta) < \lambda = \beta[(f_n)]\]
and therefore $\mu \notin \mathcal{M}_{\max}[(f_n)]$.
 \end{proof}
The proposition given below will be needed to make use of the hypothesis $\|\mathcal{A}(x,n)\| \leq M$ in the proof of Theorem \ref{onepointone}. The proof is not dissimilar to \cite[Theorem 1]{M}.
\begin{proposition}\label{P1}
Suppose that $T \colon X \to X$ is minimal. Let $(f_n)$ be a subadditive sequence of upper semi-continuous functions from $X$ to $\mathbb{R} \cup 
\{-\infty\}$. Suppose that there exists $C \in \mathbb{R}$ such that $f_n(x) \leq C$ for all $n \in \mathbb{N}$ and $x \in X$. Then either $|f_n(x)|\leq C$ for all $n \in \mathbb{N}$ and $x \in X$, or $\lim_{n \to \infty} \frac{1}{n} \sup_{x \in X} f_n(x) <0$.
\end{proposition}
 \begin{proof}
If $C<0$ then the result is trivial since $\sup f_n \leq n\sup f_1 \leq nC$ for each $n\in \mathbb{N}$, so we assume $C \geq 0$. Using Theorem \ref{SUSAET} and Lemma \ref{nonemptiness} we may take $\lambda \in \mathbb{R} \cup \{-\infty\}$ and an ergodic measure $\mu \in \mathcal{M}_T$ such that
 \[\lambda = \lim_{n \to \infty} \frac{1}{n} \sup_{x \in X} f_n(x) = \inf_{n \geq 1} \frac{1}{n} \int f_n\,d\mu.\]
Suppose that $f_N(z) < -(C+\varepsilon)<0$ for some $z \in X$ and $N,\varepsilon>0$. Using the semi-continuity of $f_N$, choose a nonempty open set $U \subseteq X$ such that $f_N(x) < -(C+\varepsilon)$ for all $x \in U$. Since $T$ is minimal we have $\mu(U)>0$. 
 
Using the Birkhoff ergodic theorem and the subadditive ergodic theorem respectively, choose $x_0 \in U$ such that $n^{-1}\sum_{k=0}^{n-1} \chi_U(T^kx_0) \to \mu(U)$ and $n^{-1}f_n(x_0) \to \lambda$. Let $(m_j)_{j=0}^\infty$ be the increasing sequence of integers given by $m_0=0$ and $m_{j+1} = \min\{m>m_j \colon T^{m_j}x_0 \in U\}$. Now let $(n_r)_{r=0}^\infty$ be given by $n_r = m_{Nr}$ so that $n_{r+1} \geq n_r + N$ and $T^{n_r}x_0 \in U$ for each $r \geq 0$. Note that $\lim_{r \to \infty}r/n_r = \mu(U)/N>0$. For each $r \in \mathbb{N}$ we have
 \[f_{n_r}(x_0) \leq\sum_{k=1}^r \left(f_N\left(T^{n_{k-1}}x_0\right) + f_{n_k-n_{k-1}-N}\left(T^{n_{k-1}+N}x_0\right)\right) \leq -r\left(C+\varepsilon\right) + rC\]
and hence
\[\lim_{n \to \infty} \frac{1}{n} \sup_{x \in X} f_n(x)=\lambda = \lim_{r \to \infty} \frac{1}{n_r} f_{n_r}(x_0) \leq -\frac{\mu(U)\varepsilon}{N}<0.\]
The proof is complete.
 \end{proof}

\section{Proof of Theorem \ref{onepointone}}\label{four} 

We require the following simple result on the metric $d_{\Gr}$.
\begin{lemma}\label{Grassmannian-metric}
Let $V, W \in \Gr(p,d)$ where $1 \leq p \leq d$. Then,
\[d_{\Gr}(V,W) = \max_{\substack{v \in V\\\|v\|=1}} \mathrm{dist}(v,W).\]
\end{lemma}
\begin{proof}
Note that $d_{\Gr}(V,W) \leq 1$ for every $V,W \in \Gr(p,d)$, see e.g. \cite[p.56]{Kato}. Let $P^\perp_V$ and $P^\perp_W$ denote the operators of orthogonal projection onto $V$ and $W$ respectively. If $\max\{\mathrm{dist}(v,W)\colon v \in V,\,\|v\|=1\}=1$ then the result is clear. Otherwise, since
\[\|(I - P^\perp_W)P^\perp_V\| = \max_{\|v\|=1}\mathrm{dist}(Pv,W) = \max_{\substack{v \in V\\\|v\|=1}} \mathrm{dist}(v,W)\]
the result follows from \cite[Theorem I-6.34]{Kato}.
\end{proof}
For each $B \in \Mat_d(\mathbb{C})$ write $|B|:=\sqrt{B^*B}$, and for $1 \leq i \leq d$ let $\sigma_1(B)\geq \ldots \geq \sigma_d(B)$ denote the eigenvalues of $|B|$ listed in decreasing order, allowing repetitions if multiplicities occur. Clearly $0 \leq \sigma_i(B) \leq\|B\|$ for every $i$. The values $\sigma_i(B)$ depend continuously on $B \in \Mat_d(\mathbb{C})$, and if $A,B \in \Mat_d(\mathbb{C})$ then for $1 \leq \ell \leq d$,
\[\prod_{i=1}^\ell \sigma_i(AB) \leq \left(\prod_{i=1}^\ell \sigma_i(A) \right)\left(\prod_{i=1}^\ell \sigma_i(B) \right),\]
see e.g. \cite{GK}. For each $x \in X$, $n \in \mathbb{N}$ and $1 \leq \ell \leq d$ let us define $f_n^\ell(x)=\sum_{i=1}^\ell \log \sigma_i(\mathcal{A}(x,n))$. Each $(f_n^\ell)$ is a subadditive sequence of continuous functions from $X$ to $\mathbb{R} \cup \{-\infty\}$ and the results of \S3 may therefore be applied.

Let $M \geq 1$ such that $\|\mathcal{A}(x,n)\| \leq M$ for all $x \in X$ and all $n \in \mathbb{N}$. For each integer $\ell$ in the range $1 \leq \ell \leq d$, define
\[\theta_\ell := \lim_{n \to \infty}\sup_{x \in X} \frac{1}{n} \sum_{i=1}^\ell \log \sigma_i( \mathcal{A}(x,n))\]
which exists by Theorem \ref{SUSAET}. For $x \in X$, $n \in \mathbb{N}$ and $1 < \ell \leq d$ we have
\[\sum_{i=1}^\ell \log \sigma_i(\mathcal{A}(x,n)) \leq \sum_{i=1}^{\ell-1} \log\sigma_i(\mathcal{A}(x,n))+ \log M \leq \ell \log M\]
and therefore $\theta_{\ell+1} \leq \theta_\ell \leq 0$ for $1 \leq \ell<d$. If $\theta_1<0$ then Theorem \ref{onepointone} is vacuously true with $p=0$, $\mathcal{V}(x) \equiv \{0\}$ and $\mathcal{W}(x)\equiv \mathbb{C}^d$, so we henceforth assume $\theta_1=0$. Take $p\in \mathbb{N}$ such that $\theta_\ell=0$ for $1 \leq \ell \leq p$ and $\theta_\ell <0$ for $p<\ell \leq d$. Applying Proposition \ref{P1} to $(f_n^\ell)$ it follows that for $1 \leq \ell \leq p$
\[-\ell \log M \leq \sum_{i=1}^\ell \log \sigma_i(\mathcal{A}(x,n))\leq \ell \log M\]
for all $x \in X$ and $n \in \mathbb{N}$. We conclude from this that there is $\delta_0>0$ such that
\[\min_{1 \leq i \leq p}\inf_{x \in X} \inf_{n \geq 1} \sigma_i(\mathcal{A}(x,n)) \geq \delta_0.\]
Since $\theta_i<0$ for $p<i \leq d$ we similarly deduce that there exist $C_0>0$, $\xi \in (0,1)$ such that for each $n \in \mathbb{N}$
\[\max_{p<i \leq d}\sup_{x \in X} \sigma_i(\mathcal{A}(x,n)) \leq C_0 \xi^n.\]

Given $x \in X$ and $n \in \mathbb{N}$, let $U^+_n(x) \in \Gr(p,d)$ be the vector space spanned by those eigenvectors of $|\mathcal{A}(x,n)|$ which correspond to the eigenvalues $\sigma_1(\mathcal{A}(x,n))$ up to $\sigma_p(\mathcal{A}(x,n))$ and let $U^-_n(x) \in \Gr(d-p,d)$ be the space spanned by those eigenvectors associated to the remaining eigenspaces. If $v$ is an eigenvector of $|\mathcal{A}(x,n)|$ with eigenvalue $\sigma_i(A(x,n))$ then
\begin{equation}\label{biscuits}\|\mathcal{A}(x,n)v\|^2 = \langle \mathcal{A}(x,n)v,\mathcal{A}(x,n)v\rangle = \langle \mathcal{A}(x,n)^*\mathcal{A}(x,n)v,v\rangle =\sigma_i(\mathcal{A}(x,n))^2\|v\|^2.\end{equation}
Since $|\mathcal{A}(x,n)|$ is a normal matrix there exists an orthonormal basis for $\mathbb{C}^d$ consisting of its eigenvectors. In particular $U^+_n(x)$ is orthogonal to $U^-_n(x)$, and using \eqref{biscuits} we may derive
\[\inf \left\{\|\mathcal{A}(x,n)v\|\colon v \in U^+_n(x)\text{ and }\|v\|=1  \right\}\geq \delta_0,\]
\[\sup\left\{\|\mathcal{A}(x,n)v\|\colon v \in U^-_n(x)\text{ and }\|v\|=1\right\}\leq C_0\xi^n\]
for all $x \in X$ and $n \in \mathbb{N}$.

We now construct the function $\mathcal{V}$ and establish its properties. The essential idea is to show for each $x\in X$ that the sequence $\mathcal{A}(x,n)U^+_{2n}(x)$ forms a Cauchy sequence in $\Gr(p,d)$ and to define $\mathcal{V}(x)$ to be its limit. This is related to the construction in \cite{FLQ}, but our argument is simplified by the presence of estimates which are uniform in $x$.

Let $x \in X$, $n\in \mathbb{N}$ and $v \in \mathbb{C}^d$; if $\|\mathcal{A}(x,n)v\| \geq \varepsilon \|v\|$ for some $\varepsilon>0$, an easy calculation shows that for $1 \leq k <n$
\begin{equation}\label{easy1}\|\mathcal{A}(x,k)v\| \geq M^{-1}\varepsilon \|v\|\end{equation}
and
\begin{equation}\label{easy2}\|\mathcal{A}(T^kx,n-k)\mathcal{A}(x,k)v\| = \|\mathcal{A}(x,n)v\|\geq M^{-1}\varepsilon\|\mathcal{A}(x,k)v\|.\end{equation}
For each $n \geq 1, \kappa>0$ and $x \in X$, define a subset of $\Gr(p,d)$ by
\[\mathfrak{V}(x,n,\kappa):=\left\{ \mathcal{A}(T^{-n}x,n)W \colon W\in \Gr(p,d),\, \|\mathcal{A}(T^{-n}x,2n)w\| \geq \kappa\|w\|\,\forall\,w \in W\right\}.\]
Note that for $\kappa \leq \delta_0 $ we have $\mathcal{A}(T^{-n}x,n)U_{2n}^+(T^{-n}x) \in \mathfrak{V}(x,n,\kappa)$ and so the latter set is nonempty. Moreover we have
\begin{equation}\label{rid1}\mathfrak{V}(x,n,\kappa) \subseteq \mathfrak{V}\left(x,k,M^{-2}\kappa\right)\end{equation}
for $x \in X$ and $1 \leq k \leq n$ and
\begin{equation}\label{rid2}\mathcal{A}(x) \mathfrak{V}(x,n,\kappa) \subseteq \mathfrak{V}\left(x,n-1,M^{-1}\kappa\right)\end{equation}
 by virtue of \eqref{easy1} and \eqref{easy2}. We claim that for each $n \in \mathbb{N}$ and $x \in X$,
\begin{equation}\label{diameter1}\diam \bigcup_{r =n}^\infty \mathfrak{V}(x,r,\kappa) \leq \kappa^{-1}C_1\xi^n\end{equation}
where $C_1:=2C_0M$. Suppose that
\[\mathcal{A}\left(T^{-(n+m)}x,n+m\right)w \in \mathcal{A}\left(T^{-(n+m)}x,n+m\right)W \in \mathfrak{V}(x,n+m,\kappa).\]
Let $P$ be given by orthogonal projection from $\mathbb{C}^d$ onto $U_n^+(T^{-n}x)$. We have
\[\left\|\mathcal{A}\left(T^{-{(n+m)}}x,n+m\right)w-\mathcal{A}\left(T^{-n}x,n\right)P\mathcal{A}\left(T^{-(n+m)}x,m\right)w\right\|\]
\[\leq C_0\xi^n\left\|\mathcal{A}\left(T^{-(n+m)}x,m\right)w\right\| \leq C_0M\xi^n\left\|\mathcal{A}\left(T^{-(n+m)}x,n+m\right)w\right\|\]
where we have used \eqref{easy2}. It follows that
\[\mathrm{dist}(v,\mathcal{A}(T^{-n}x,n)U_n^+(x)) \leq MC_0\kappa^{-1}\xi^n\|v\|\]
for all $v \in \mathcal{A}\left(T^{-(n+m)}x,n+m\right)W$ and therefore \eqref{diameter1} holds by Lemma \ref{Grassmannian-metric}. We deduce that for each $x \in X$ the set
\[\bigcap_{n=1}^\infty \overline{\bigcup_{r=n}^\infty \mathfrak{V}(x,r,\kappa)}\]
contains a unique element for each $\kappa \leq \delta_0$. Since clearly $\mathfrak{V}(x,n,\kappa_1) \subseteq \mathfrak{V}(x,n,\kappa_2)$ for $\kappa_2 \leq \kappa_1$ it follows that this element does not depend on $\kappa$. Denote this element by $\mathcal{V}(x)$. We have $\mathcal{A}(x)\mathcal{V}(x)=\mathcal{V}(Tx)$ as an easy consequence of \eqref{rid2}. Now take $n$ large enough that $\delta_0^{-1}C_1\xi^n<\delta_0/3M$ and let $P$ be given by orthogonal projection onto some arbitrarily selected element of $\mathfrak{V}(x,n,\delta_0)$. Given any $v \in \mathcal{V}(x)$, we have $\|v-Pv\| \leq (\delta_0/3M)\|v\|$ as a consequence of \eqref{diameter1}. In particular this implies $\|Pv\| \geq (2/3)\|v\|$. We have
\[\|\mathcal{A}(x,n)v\| \geq \|\mathcal{A}(x,n)Pv\| - \|\mathcal{A}(x,n)v - \mathcal{A}(x,n)Pv\|\]\[ \geq \delta_0\|Pv\| - M\|v-Pv\| \geq (\delta_0/3)\|v\|.\]
It follows from \eqref{easy1} that for all $x \in X$ and every $n \in \mathbb{N}$ we have $\|\mathcal{A}(x,n)v\| \geq (\delta_0/3M)\|v\|$ for every $v \in \mathcal{V}(x)$.

It remains to show that $\mathcal{V}(x)$ depends continuously on $x$. Define $\delta:= \delta_0/3M$. Let $n \in \mathbb{N}$ and suppose that $x,y \in X$ satisfy
\[\max\left\{\|\mathcal{A}(T^{-n}x,2n)-\mathcal{A}(T^{-n}y,2n)\|, \|\mathcal{A}(T^{-n}x,n)-\mathcal{A}(T^{-n}y,n)\|\right\} \leq \delta\xi^n.\]
If $w \in \mathcal{V}(T^{-n}x)$ then
\[\|\mathcal{A}(T^{-n}y,2n)w\| \geq \|\mathcal{A}(T^{-n}x,2n)w\| - \delta\xi^n\|w\| \geq (1-\xi)\delta\|w\|\]
and it follows that $\mathcal{A}(T^{-n}y,n)\mathcal{V}(T^{-n}x) \in \mathfrak{V}(y,n,(1-\xi)\delta)$. If $v = \mathcal{A}(T^{-n}x,n)w \in \mathcal{V}(x) = \mathcal{A}(T^{-n}x,n)\mathcal{V}(T^{-n}x)$, then
\[\|\mathcal{A}(T^{-n}x,n)w - \mathcal{A}(T^{-n}y,n)w\| \leq \delta\xi^n\|w\| \leq \xi^n\|v\|.\]
It follows from Lemma \ref{Grassmannian-metric} that $d_{\Gr}(\mathcal{V}(x),\mathcal{A}(T^{-n}y,n)\mathcal{V}(T^{-n}x)) \leq \xi^n$, and therefore $d_{\Gr}(\mathcal{V}(x),\mathcal{V}(y)) \leq (1+C_1(1-\xi)^{-1}\delta^{-1})\xi^n$ as required.

We next construct the function $\mathcal{W}$ and establish its properties. Similarly to the case of $\mathcal{V}$, the idea is to show that $U_n^-(x)$ forms a Cauchy sequence and to define $\mathcal{W}(x)$ to be its limit. This section of the proof thus more closely approaches certain proofs of the Oseledec multiplicative ergodic theorem such as that given in \cite{Ruelle}, though as before we differ from the measurable case in that we require uniform estimates.

For each $n \in \mathbb{N}$, $x \in X$ and $K>0$ define
\[\mathfrak{W}(x,n,K) = \left\{W \in \Gr(d-p,d) \colon \|\mathcal{A}(x,n)v\| \leq K\xi^n\|v\|\text{ for all }v \in W\right\}.\]
Note that $U_n^-(x) \in \mathfrak{W}(x,n,K)$ for every $K \geq C_0$ and in particular $\mathfrak{W}(x,n,K)$ is nonempty. We assert that for each $n \in \mathbb{N}$ we have
\begin{equation}\label{Techdiam}\diam \bigcup_{r=n}^\infty \mathfrak{W}(x,r,K) \leq KC_2\xi^n\end{equation}
where $C_2 = 2\delta_0^{-1}(M+1)$. Suppose that $r\geq n\in \mathbb{N}$ and $W \in \mathfrak{W}(x,r,K)$ are given, and let $v \in W$. Write $v=u_1+u_2$ with $u_1 \in U^+_n(x)$ and $u_2 \in U^-_n(x)$; we have $\|u_2\| \leq \|v\|$ since the two spaces are mutually orthogonal. Since
\[\|\mathcal{A}(x,r)(v-u_2)\| \leq K\xi^{r}\|v\| + MC_0\xi^n\|u_2\| \leq (MC_0+K)\xi^n\|v\|\]
and
\[\|\mathcal{A}(x,r)(v-u_2)\| = \|\mathcal{A}(x,r)u_1\| \geq \delta_0\|u_1\|\]
we have
\[\mathrm{dist}(v,U^-_n(x)) = \|u_1\| \leq \delta^{-1}_0(MC_0+K)\xi^n\|v\| \leq \delta_0^{-1}(M+1)K\xi^n\|v\|\]
and \eqref{Techdiam} holds by Lemma \ref{Grassmannian-metric}. It follows that for each $x \in X$ and $K \geq C_0$, the set
\[\bigcap_{n=1}^\infty \overline{\bigcup_{r=n}^\infty \mathfrak{W}(x,r,K)}\]
contains a unique element, which we denote by $\mathcal{W}(x)$. Since clearly $\mathfrak{W}(x,n,K_1) \subseteq \mathfrak{W}(x,n,K_2)$ when $K_1 \leq K_2$  the definition of $\mathcal{W}(x)$ is not influenced by the choice of $K \geq C_0$. Given $x \in X$ and $n \in \mathbb{N}$, let $P$ be given by orthogonal projection onto $U^-_n(x) \in \mathfrak{W}(x,n,C_0)$. For each $w \in \mathcal{W}(x)$ we have
$\|w-Pw\| \leq C_0C_2\xi^n\|w\|$ as a consequence of \eqref{Techdiam} and hence
\[\|\mathcal{A}(x,n)w\| \leq \|\mathcal{A}(x,n)(w-Pw)\| + \|\mathcal{A}(x,n)Pw\|\]\[\leq MC_0C_2\xi^n\|w\| + C_0\xi^n\|Pw\| \leq (MC_2+1)C_0\xi^n\|w\|\]
as required for the statement of Theorem \ref{onepointone}.

We next prove that $\mathcal{W}(x)$ depends continuously on $x$. Let $x,y \in X$ and suppose that $d(x,y)$ is small enough that $\|\mathcal{A}(x,n)-\mathcal{A}(y,n)\| \leq \xi^n$. Since for any $w \in \mathcal{W}(x)$,
\[\|\mathcal{A}(y,n)w\| \leq \xi^n\|w\| + \|\mathcal{A}(x,n)w\| \leq (MC_2C_0+C_0+1)\xi^n\|w\|,\]
we have $\mathcal{W}(x) \in \mathfrak{W}(y,n,MC_2C_0+C_0+1)$ and it follows from \eqref{Techdiam} that
\[d_{\Gr}(\mathcal{W}(x),\mathcal{W}(y)) \leq (MC_2C_0+C_0+1)C_2\xi^n.\]
The following standard argument shows that $\mathcal{W}$ is invariant. For each $x \in X$ define
\[\tilde{\mathcal{W}}(x) = \left\{v \in \mathbb{C}^d \colon \limsup_{n \to \infty} \|A^n_xv\|^{1/n} <1\right\}.\]
Clearly $\tilde{\mathcal{W}}(x)$ is a linear subspace of $\mathbb{C}^d$, $\mathcal{W}(x) \subseteq \tilde{\mathcal{W}}(x)$, and $\mathcal{A}(x)\tilde{\mathcal{W}}(x) \subseteq\tilde{\mathcal{W}}(Tx)$. If $\dim \tilde{\mathcal{W}}(x)>\dim \mathcal{W}(x)$ then $\tilde{\mathcal{W}}(x) \cap \mathcal{V}(x) \neq \{0\}$ which clearly entails a contradiction. It follows that $\tilde{\mathcal{W}}(x) = \mathcal{W}(x)$ for all $x \in Z$ and therefore $\mathcal{A}(x)\mathcal{W}(x) =\mathcal{A}(x)\tilde{\mathcal{W}}(x) \subseteq \tilde{\mathcal{W}}(Tx) = \mathcal{W}(Tx)$, which concludes our study of the properties of $\mathcal{W}$.

For each $x \in X$ let $P(x)$ denote the projection having image $\mathcal{V}(x)$ and kernel $\mathcal{W}(x)$. It remains to prove that $P(x)$ depends continuously on $x$. We will show that for every $x \in X$, if $y$ satisfies
\begin{equation}\label{bargage}3\|P(x)\|.[d_{\Gr}(\mathcal{V}(x),\mathcal{V}(y)) + d_{\Gr}(\mathcal{W}(x),\mathcal{W}(y))]<\frac{1}{2}\end{equation}
then
\begin{equation}\label{bleh}\|P(x)-P(y)\| \leq 12\|P(x)\|.[d_{\Gr}(\mathcal{V}(x),\mathcal{V}(y)) + d_{\Gr}(\mathcal{W}(x),\mathcal{W}(y))].\end{equation}
Since $X$ is compact we may deduce that $\sup\|P\|$ is finite and the result follows. 

 For notational convenience we write $Q(x)=I-P(x)$ for all $x \in X$. Fix $x \in X$, and for each $y \in X$ define $U(x,y) = P^{\perp}_{\mathcal{V}(y)} P(x) + P^{\perp}_{\mathcal{W}(y)}Q(x)$, where $P^{\perp}_Z$ denotes orthogonal projection onto $Z$. Since $I = P(x)+Q(x)=P^\perp_{\mathcal{V}(x)}P(x) + P^\perp_{\mathcal{W}(x)}Q(x)$ we have
\begin{equation}\label{icky}\left\|U(x,y)-I\right\| \leq (2\|P(x)\|+1).\left[d_{\Gr}(\mathcal{V}(x),\mathcal{V}(y)) + d_{\Gr}(\mathcal{W}(x),\mathcal{W}(y))\right].\end{equation}
Suppose that $y$ satisfies \eqref{bargage}. Then $U(x,y)$ is invertible and
\begin{equation}\label{ecky}\|U(x,y)^{-1}-I \| \leq \sum_{n=1}^\infty \|U(x,y)-I\|^n \end{equation}
\[\leq 6\|P(x)\|\left[d_{\Gr}(\mathcal{V}(x),\mathcal{V}(y)) + d_{\Gr}(\mathcal{W}(x),\mathcal{W}(y))\right].\]
Since for each $v \in \mathcal{V}(x)$ and $w \in \mathcal{W}(x)$ we have
\[U(x,y)P(x)(v+w) = U(x,y)v =P(y)U(x,y)(v+w)\]
it follows that $P(y) = U(x,y)P(x)U(x,y)^{-1}$. Combining this with \eqref{icky} and \eqref{ecky} yields \eqref{bleh} and the proof is complete.

\section{Proof of Theorem \ref{Tech}}\label{five}

Let $\mathsf{A}\subset \Mat_d(\mathbb{C})$ be compact and product bounded with $\varrho(\mathsf{A})=1$, let $\vvv\cdot\vvv$ be an extremal norm for $\mathsf{A}$, and choose $M>0$ such that $\vvv v \vvv \leq M\|v\| \leq M^2\vvv v \vvv$ for all $v \in \mathbb{C}^d$. As in the introduction we let $\mathcal{A} \colon \mathsf{A}^{\mathbb{Z}} \to \Mat_d(\mathbb{C})$ be given by projection onto the zeroth co-ordinate, let $T \colon \mathsf{A}^{\mathbb{Z}} \to \mathsf{A}^{\mathbb{Z}}$ be given by the shift map, and take $d$ to be the metric on $\mathsf{A}^{\mathbb{Z}}$ defined previously. Clearly $\mathcal{A}$ and $T$ are Lipschitz continuous. For each $n \in \mathbb{N}$ we have
\[\left\{\mathcal{A}(x,n) \colon x \in \mathsf{A}^{\mathbb{Z}} \right\} = \left\{A_n\cdots A_1 \colon A_i \in \mathsf{A}\right\}\]
and therefore $\sup \left\{\log \vvv \mathcal{A}(x,n)\vvv \colon x \in \mathsf{A}^{\mathbb{Z}}\right\}=0$ for all $n \in \mathbb{N}$. By Lemma \ref{subordinationsprinzip} the set
\[Y :=\{x \colon \vvv\mathcal{A}(x,n)\vvv=1\text{ for all }n \geq 1\}\]
is compact and nonempty and satisfies $TY \subseteq Y$.

 Let $Z = TZ$ be any minimal set contained in $Y$. Note that for all $x \in Z$ and $n \in \mathbb{N}$ we have $\| \mathcal{A}(x,n)\|  \leq M^2$ since $\vvv \mathcal{A}(x,n)\vvv=1$.  We may therefore apply Theorem \ref{onepointone} to the minimal set $Z$ and the cocycle $\mathcal{A}$. If $p=0$ then we would have $\vvv\mathcal{A}(x,n)\vvv<1$ for some $x \in Z$ and $n \in \mathbb{N}$, so it must be the case that $p \in \mathbb{N}$. To prove Theorem \ref{Tech}, we must show firstly that the functions $\mathcal{V},\mathcal{W}$ and $P$ provided by Theorem \ref{onepointone} are H\"older continuous, and secondly that for all $x\in Z$ and $n\in \mathbb{N}$ one has $\vvv\mathcal{A}(x,n)v\vvv=\vvv v\vvv$ for every $v \in \mathcal{V}(x)$.

The proof of H\"older continuity is straightforward. Let $\delta, \xi$ be as given by Theorem \ref{onepointone}. Given any $\varepsilon>0$, choose $C_\varepsilon>0$ such that $2C_\varepsilon n e^{-n\varepsilon} \leq 1$ for all $n \in \mathbb{N}$.  If $d(x,y) \leq C_\varepsilon\delta M^{-4} e^{-n\varepsilon}\xi^{n} 2^{-n}$ then
\[\max_{-n \leq k \leq n} \left\|\mathcal{A}(T^kx)-\mathcal{A}(T^ky)\right\| \leq 2^n \sum_{i \in \mathbb{Z}} \frac{\left\|\mathcal{A}(T^ix)-\mathcal{A}(T^iy)\right\|}{2^{|i|}} \leq C_\varepsilon\delta M^{-4}\xi^{n}e^{-n\varepsilon}\]
and therefore
\[\left\|\mathcal{A}(T^{-n}x,2n)-\mathcal{A}(T^{-n}y,2n)\right\|\]
\[ \leq \sum_{i=-n}^{n-1} \left\|\mathcal{A}(T^{i+1}x,n-i-1)\right\|.\left\|\mathcal{A}(T^ix)-\mathcal{A}(T^iy)\right\|.\left\|\mathcal{A}(T^{-n}y,n+i)\right\|\]\[ \leq 2C_\varepsilon n\delta \xi^{n}e^{-\varepsilon n} \leq \delta \xi^n\]
where we adopt the convention $\mathcal{A}(\cdot,0)\equiv I$. The same estimate clearly also yields $\|\mathcal{A}(T^{-n}x,n)-\mathcal{A}(T^{-n}y,n)\| \leq \delta\xi^n$ and $\|\mathcal{A}(x,n)-\mathcal{A}(y,n)\| \leq \delta \xi^n \leq \xi^n$. Applying Theorem \ref{onepointone} we deduce that  $d_{\Gr}(\mathcal{V}(x),\mathcal{V}(y))$ and $d_{\Gr}(\mathcal{W}(x),\mathcal{W}(y))$ are both bounded by $\tilde C \xi^n$. It follows that for $\alpha := \log \xi / (\log \xi - \log 2-\varepsilon)>0$,
\[\sup_{\substack{x,y \in Z\\x \neq y}} \frac{d_{\Gr}(\mathcal{V}(x),\mathcal{V}(y))}{d(x,y)^\alpha} < \infty\]
and similarly for $\mathcal{W}$ so that $\mathcal{V}$ and $\mathcal{W}$ are both $\alpha$-H\"older continuous. By Theorem \ref{onepointone} this implies that $P$ is $\alpha$-H\"older continuous also.

We now prove that for every $x \in Z$ and $n \in \mathbb{N}$ we have $\vvv \mathcal{A}(x,n)v\vvv = \vvv v \vvv$ for every $v \in \mathcal{V}(x)$. For each $x \in Z$ define
\[\mathcal{S}(x) :=\left\{B \colon \liminf_{n\to \infty} \max\big[d(T^nx,x),\vvv\mathcal{A}(x,n)-B\vvv\big] =0\right\}.\]
Since $T$ acts minimally on $Z$, $x$ is recurrent, and since $\vvv \mathcal{A}(x,n)\vvv=1$ for each $n$ the set $\mathcal{S}(x)$ is nonempty. If $\lim_{k \to \infty}B_k = B$ with each $B_k \in \mathcal{S}(x)$ then we may choose a strictly increasing sequence $(n_k)$ such that $d(T^{n_k}x,x)<1/k$, $\vvv B_k-B\vvv \leq 1/k$ and $\vvv \mathcal{A}(x,n_k)-B_k \vvv <1/k$ for each $k \in \mathbb{N}$, which shows that $B \in \mathcal{S}(x)$ and therefore $\mathcal{S}(x)$ is closed. Since clearly $\vvv B \vvv = 1$ for all $B \in \mathcal{S}(x)$ it follows that $\mathcal{S}(x)$ is compact.

We claim that $\mathcal{S}(x)$ is a semigroup. Let $B_1,B_2 \in \mathcal{S}(x)$; it suffices to show that for any $N,\varepsilon>0$ there is $n>N$ such that $d(T^nx,x)<\varepsilon$ and $\vvv \mathcal{A}(x,n)-B_1B_2\vvv<\varepsilon$. Since $B_1 \in \mathcal{S}(x)$ we can choose $n_1>N$ such that $\vvv\mathcal{A}(x,n_1)-B_1\vvv <\varepsilon/3$ and $d(T^{n_1}x,x)<\varepsilon/2$. Since $B_2 \in \mathcal{S}(x)$ we may choose $n_2>N$ such that $\vvv \mathcal{A}(x,n_2)-B_2 \vvv < \varepsilon/3$ and such that $d(T^{n_2}x,x)$ is so small as to guarantee $\vvv \mathcal{A}(T^{n_2}x,n_1)-\mathcal{A}(x,n_1)\vvv < \varepsilon/3$ and $d(T^{n_1+n_2}x,T^{n_1}x)<\varepsilon/2$. We have
\begin{align*}
\vvv\mathcal{A}(x,n_1+n_2)-B_1B_2\vvv &\leq \vvv\mathcal{A}(T^{n_2}x,n_1)\mathcal{A}(x,n_2) -\mathcal{A}(x,n_1)\mathcal{A}(x,n_2)\vvv\\&+\vvv\mathcal{A}(x,n_1)\mathcal{A}(x,n_2) - \mathcal{A}(x,n_1)B_2 \vvv\\
&+ \vvv\mathcal{A}(x,n_1)B_2 - B_1B_2\vvv<\varepsilon\end{align*}
and
\[d(T^{n_1+n_2}x,x) \leq d(T^{n_1+n_2}x,T^{n_1}x) + d(T^{n_1}x,x) < \varepsilon\]
as required to prove the claim.

Given any $B \in \mathcal{S}(x)$, take $(n_r)_{r=1}^\infty$ such that $\mathcal{A}(x,n_r) \to B$ and $d(T^{n_r}x,x) \to 0$. If $v$ is a nonzero element of $\mathcal{V}(x)$ then clearly $\mathcal{A}(x,n_r)v \to Bv$. Since Theorem \ref{onepointone} gives $\|\mathcal{A}(x,n_r)v\| \geq \delta \| v \|$ for all $r \in \mathbb{N}$ we have $\|Bv\| \geq \delta \| v \|>0$. Since $\mathcal{A}(x,n_r)\mathcal{V}(x)=\mathcal{V}(T^{n_r}x)$, $T^{n_r}x \to x$ and $\mathcal{V}$ is continuous it follows that in fact $Bv$ is a nonzero element of  $\mathcal{V}(x)$. By a  simpler version of the same argument we see that $Bw=0$ for every $w \in \mathcal{W}(x)$, and we conclude that the image of $B$ is precisely $\mathcal{V}(x)$ whilst the kernel of $B$ is precisely $\mathcal{W}(x)$. 

We now finish the proof. Since $\mathcal{S}(x)$ is a compact semigroup, it contains an idempotent element $P$ (see e.g. \cite{HM}). If $\vvv \mathcal{A}(x,k)v\vvv \leq (1-\varepsilon)\vvv v \vvv$ for some vector $v \in \mathcal{V}(x)$ and positive integer $k$, then $\vvv \mathcal{A}(x,n)v\vvv \leq (1-\varepsilon)\vvv v \vvv$ for all large enough $n$ and therefore $\vvv Pv \vvv \leq (1-\varepsilon)\vvv v \vvv$. But since $v$ lies in the image of $P$ we have $v = Pw = P^2w = Pv$ for some $w \in \mathbb{C}^d$, and we conclude that $\vvv v\vvv$ must equal zero. It follows that for each $x \in X$ and $n \in \mathbb{N}$ we have $\vvv \mathcal{A}(x,n)v\vvv = \vvv v \vvv$ for all $v \in \mathcal{V}(x)$ and the theorem is proved.

\begin{remark} Since we have identified both the image and the kernel of the idempotent $P$, it follows that for each $x$ the semigroup $\mathcal{S}(x)$ in fact contains a \emph{unique} idempotent element, namely the projection $P(x)$. The family of semigroups $\mathcal{S}(x)$ should be compared to the ``limit semigroup'' introduced by F. Wirth \cite{Wirth}.\end{remark}


 \section{Proof of Theorem \ref{QBWF}}\label{six}
 The following lemma allows us to ignore cases in which $\mathsf{A}$ fails to be product bounded. Results of this kind are used in the proofs of Theorem \ref{BWF} given by Berger-Wang \cite{BW}, Elsner \cite{E}, and Shih \emph{et al.} \cite{SWP}.
\begin{lemma}\label{redux}
Let $\mathsf{A}\subset \Mat_d(\mathbb{C})$ be bounded set such that $\varrho(\mathsf{A})=1$ and $\mathsf{A}$ is not product bounded. Then exist a positive integer $d' < d$ and $U \in GL_d(\mathbb{C})$ such that if $P$ denotes the natural projection from $\mathbb{C}^d$ to $\mathbb{C}^{d'}$ then the set $\hat{\mathsf{A}}:=PU^{-1}\mathsf{A}U$ satisfies $\varrho(\mathsf{A})=1$, is product bounded and satisfies $\varrho_n^-(\mathsf{A}) \geq \varrho^-_n(\hat{\mathsf{A}})$ for each $n \in \mathbb{N}$.
\end{lemma}
\begin{proof}
Using \cite[Lemma 4]{E} we can find $\hat{\mathsf{A}}=PU^{-1}\mathsf{A}U$ which satisfies all of the required properties except possibly for product boundedness. By repeating this procedure we either obtain a product bounded $\hat{\mathsf{A}}$ with $d'>1$, or reduce to the case $d=1$ in which case product boundedness is satisfied automatically.
\end{proof}
If $\varrho(\mathsf{A})=0$ then we have nothing to prove, and if $\varrho(\mathsf{A})>0$ then by normalising $\mathsf{A}$ if necessary we may assume that $\varrho(\mathsf{A})=1$. To prove Theorem \ref{QBWF}, therefore, it suffices by Lemma \ref{redux} to assume that $\mathsf{A}$ is a finite set of $d \times d$ matrices such that $\varrho(\mathsf{A})=1$ and $\mathsf{A}$ is product bounded. Since $\mathsf{A}$ is finite, the metric described in the introduction is Lipschitz equivalent to the more easily-used metric given by
\[d\left[(A_i)_{i \in \mathbb{Z}},(B_i)_{i \in \mathbb{Z}}\right] = 2^{-\sup\{n \geq 0 \colon A_i = B_i \text{ for } |i|  \leq n\}}.\]

The following proposition may be obtained easily by modifying a result of X. Bressaud and A. Quas \cite[Theorem 1]{BQ}.
\begin{proposition}\label{urghurghurgh}
Let $\mathsf{A}$ be finite, let $Z \subseteq \mathsf{A}^{\mathbb{Z}}$ be compact with $TZ=Z$, and let $N \in \mathbb{N}$. Then there exist sequences of integers $(r_n)$, $(m_n)$ and a sequence of points $x_n \in \mathsf{A}^{\mathbb{Z}}$ such that $m_n^{-1} \log n\to 0$ and such that for all sufficiently large $n$ each $r_n$ is divisible by $N$, $r_n \leq n$, $T^{r_n}x_n=x_n$ and
\[\max_{0 \leq k < r_n} d(T^kx_n,Z) \leq 2^{-m_n}.\]
\end{proposition}
Now let $\vvv\cdot \vvv$ be an extremal norm for $\mathsf{A}$, let $Y$ be as in Theorem \ref{Tech}, and let $Z \subseteq Y$ be any minimal set. Let $\mathcal{V}$, $\mathcal{W}$, $P$, $C$ and $\xi$ be as given by Theorem \ref{Tech}, and define $Q(x)=I-P(x)$ for each $x \in Z$. Note that for $v \in \mathcal{V}(x)$ and $w \in \mathcal{W}(x)$ we have
\[\mathcal{A}(x,n)P(x)(v+w) = \mathcal{A}(x,n)v = P(T^nx)\mathcal{A}(x,n)(v+w)\]
and therefore $\mathcal{A}(x,n)P(x) = P(T^nx) \mathcal{A}(x,n)$ for all $x \in Z$ and $n \in \mathbb{N}$. Clearly this implies that $\mathcal{A}(x,n)Q(x)=Q(T^nx)\mathcal{A}(x,n)$ for all $x \in Z$ and $n \in \mathbb{N}$.

The following two lemmas, and the general strategy of their application, are suggested by \cite{K}. For each $x \in Z$ and $\theta>0$ let us define
\[\mathfrak{C}(x,\theta)=\left\{v \in \mathbb{C}^d \colon \theta\vvv P(x)v\vvv\geq \vvv Q(x)v\vvv\right\}.\]
\begin{lemma}\label{cone1}
Let $x,y \in Z$ and suppose that $\vvv P(x)-P(y)\vvv \leq \theta<1/5$. Then $\mathfrak{C}(x,\theta) \subseteq \mathfrak{C}(y,3\theta)$.
\end{lemma}
\begin{proof}
If $v \notin \mathfrak{C}(y,3\theta)$ then $\vvv Q(y) v\vvv>3\theta \vvv P(y)v\vvv$ and therefore
\[3\theta\vvv P(x)v\vvv \leq 3\theta\vvv P(y)\vvv + 3\theta^2 \vvv v \vvv <  \vvv Q(y)v\vvv + 3\theta^2 \vvv v \vvv \leq \vvv Q(x)v\vvv + (\theta+3\theta^2)\vvv v\vvv\]
\[\leq (1+\theta+3\theta^2)\vvv Q(x)v\vvv + (\theta+3\theta^2)\vvv P(x)v\vvv\]
and therefore
\[\theta \vvv P(x)v\vvv \leq \frac{2\theta - 3\theta^2}{1+\theta+3\theta^2}\vvv P(x)v\vvv<\vvv Q(x)v\vvv\]
so that $v \notin \mathfrak{C}(x,\theta)$. 
\end{proof}
\begin{lemma}\label{cone2}
Let $x \in Z$ and $n \in \mathbb{N}$, and suppose that $v \in \mathfrak{C}(x,\theta)$ for some $\theta \in (0,1]$. Then $\mathcal{A}(x,n)v \in \mathfrak{C}(T^nx,K_1\xi^n\theta)$ and $\vvv \mathcal{A}(x,n)v\vvv \geq (1-\theta - K_1\xi^n\theta)\vvv v \vvv$, where $K_1>0$ does not depend on $x$, $n$, $\theta$ or $v$.
\end{lemma}
\begin{proof}
Let $M=\sup_{z \in Z}\vvv Q(x)\vvv$ and $K_1=2CM$. If $v \in \mathfrak{C}(x,\theta)$ then clearly
\[\vvv v\vvv \leq \vvv P(x)v\vvv +\vvv Q(x)v\vvv \leq (1+\theta)\vvv P(x)v\vvv.\]
Using Theorem \ref{Tech} it follows that
\[\vvv P(T^nx)\mathcal{A}(x,n)v\vvv = \vvv\mathcal{A}(x,n)P(x)v\vvv = \vvv P(x)v\vvv \geq (1+\theta)^{-1}\vvv v\vvv\]
and
\[\vvv Q(T^nx)\mathcal{A}(x,n)v\vvv =\vvv\mathcal{A}(x,n)Q(x)v\vvv \leq C_1\xi^n\vvv Q(x)v\vvv \leq CM\theta\xi^n\vvv v\vvv.\]
Consequently
\[\vvv\mathcal{A}(x,n)v\vvv \geq \vvv P(T^nx)\mathcal{A}(x,n)v\vvv - \vvv Q(T^nx)\mathcal{A}(x,n)v\vvv\geq \left(1-\theta - K_1\theta\xi^n\right)\vvv v\vvv\]
and
\[\vvv Q(T^nx)\mathcal{A}(x,n)v\vvv \leq K_1\xi^n\theta\vvv P(T^nx)\mathcal{A}(x,n)v \vvv\]
as required. 
\end{proof}
We now prove Theorem \ref{QBWF}. Let $K_2,\alpha>0$ such that $\vvv P(x)-P(y)\vvv \leq K_2d(x,y)^\alpha$ for all $x,y \in Z$, let $N\geq1$ be large enough that $K_1\xi^N<1/3$, and let $(x_n),(m_n),(r_n)$ be as given by Proposition \ref{urghurghurgh}. Suppose that $n$ is large enough that $K_2 2^{\alpha(N-m_n)}<1/5$,  $m_n \geq N$, and all of the properties listed in Proposition \ref{urghurghurgh} are satisfied. Let $q=r_n/N$ and choose $z_1,\ldots,z_q$ such that $d(z_i,T^{(i-1)N}x_n) \leq 2^{-m_n}$ for each $i$. We then have
\[d(T^Nz_i,z_{i+1}) = \max\{d(T^Nz_1,T^{iN}x_n),d(T^{iN}x_n,z_{i+1})\}\leq 2^{N-m_n}\]
for $1 \leq i <q$, and similarly $d(T^N z_q,z_1)\leq 2^{N-m_n}$. If $v \in \mathfrak{C}(z_i,K_22^{\alpha(N-m_n)})$ for $1 \leq i < q$ then we may apply Lemmas \ref{cone1} and \ref{cone2} to deduce that $\mathcal{A}(z_i,N)v \in \mathfrak{C}(z_{i+1},K_22^{\alpha(N-m_n)})$ and $\vvv \mathcal{A}(z_i,N)v\vvv \geq (1-K_22^{1+\alpha(N-m_n)})\vvv v \vvv$, and similarly if $v \in \mathfrak{C}(z_q,K_22^{\alpha(N-m_n)})$ then $\mathcal{A}(z_q,N)v \in \mathfrak{C}(z_1,K_22^{\alpha(N-m_n)})$ and $\vvv \mathcal{A}(z_q,N)v\vvv \geq (1-K_22^{1+\alpha(N-m_n)})\vvv v \vvv$. It follows that if $v \in \mathfrak{C}(z_1,K_22^{\alpha(N-m_n)})$ then
\[\mathcal{A}(x_n,r_n)v = \mathcal{A}(z_q,N)\cdots \mathcal{A}(z_1,N)v \in \mathfrak{C}(z_1,K_22^{\alpha(N-m_n)})\]
(where we have used $m_n \geq N$) and
\[\vvv \mathcal{A}(x_n,r_n)v\vvv = \vvv\mathcal{A}(z_q,N)\cdots \mathcal{A}(z_1,N)v \vvv \geq (1-K_22^{1+\alpha(N-m_n)})^{r_n/N}\vvv v \vvv.\]
If we choose $v \in \mathfrak{C}(z_1,K_22^{\alpha(N-m_n)})$ with $\vvv v \vvv=1$, then since $r_n \leq n$ we deduce
\begin{align*}
\max_{1 \leq k \leq n} \varrho^-_k(\mathsf{A}) \geq \rho(\mathcal{A}(x_n,r_n))^{1/r_n} &= \left(\lim_{k \to \infty} \vvv \mathcal{A}(x_n,r_n)^k\vvv^{1/k}\right)^{1/r_n}\\ &\geq \left(\liminf_{k \to \infty} \vvv \mathcal{A}(x_n,r_n)^k v \vvv^{1/k}\right)^{1/r_n}\\& \geq (1-K_22^{1+\alpha(N-m_n)})^{1/N} \geq 1-K_22^{1+\alpha(N-m_n)}.\end{align*}
It follows that for all large enough $n$
\[\left|\varrho(\mathsf{A}) - \max_{1 \leq k \leq n} \varrho_k^-(\mathsf{A})\right| \leq \left(K_2 2^{1+\alpha N}\right)2^{-\alpha m_n}.\]
To complete the proof we have only to observe that the condition $m_n^{-1} \log n\to 0$ is equivalent to the assertion that $e^{-\varepsilon m_n} = O(1/n^r)$ for every $r,\varepsilon>0$.

\section{Discussion on possible extensions of Theorem \ref{QBWF}}\label{seven}

We shall now briefly discuss some of the limitations of the method of proof of Theorem \ref{QBWF} and the prospects for an extension of that theorem using the approach of the present article. 

Fix some compact set $\Omega \subset \mathbb{C}^d$, and consider the metric space $\Omega^{\mathbb{Z}}$ equipped with the metric $d[(x_i),(y_i)] = \sum_{i \in \mathbb{Z}} 2^{-|i|}\|x_i-y_i\|$ together with the shift map $T \colon \Omega^{\mathbb{Z}} \to \Omega^{\mathbb{Z}}$. Given a compact $T$-invariant set $Z \subseteq \Omega^{\mathbb{Z}}$, let us define
\[\varepsilon(Z,n) = \min_{1 \leq k \leq n} \inf_{T^kx=x} \max_{0 \leq i <k} \mathrm{dist}(T^ix,Z).\]
The magnitude of the error term in the proof of Theorem \ref{QBWF} is determined by the result of X. Bressaud and A. Quas in \cite{BQ} which asserts that if $\Omega$ is a finite set, then $\varepsilon(Z,n)=O(1/n^r)$ for every $r \in \mathbb{N}$. (To simplify our proof we in fact considered only approximations using periodic orbits whose period is divisible by $N$, but this requirement could be dispensed with without difficulty.) Bressaud and Quas' result is essentially sharp: see \cite{BQ} and related work in \cite{CM}. In the case where $\Omega$ is compact but not finite, the rate of decrease of $\varepsilon(Z,n)$ can be much slower, and this is the principal obstacle in extending Theorem \ref{QBWF} to the case in which $\mathsf{A}$ compact but infinite. The following simple example illustrates the problem.

Suppose that $\Omega = S^1 \subset \mathbb{C}$. Let $\gamma = (1-\sqrt{5})/2$ and define
\[Z = \left\{\left(e^{2\pi im \gamma} \omega\right)_{m \in \mathbb{Z}} \colon \omega \in S^1\right\},\]
which is clearly compact and $T$-invariant. Let $n \in \mathbb{N}$ and $1 \leq k \leq n$, and suppose that $x \in \Omega^{\mathbb{Z}}$ has $T^kx=x$ and $\max_{0 \leq j <k}\mathrm{dist}(T^jx,Z)\leq 2\varepsilon(Z,n)$. For $j=0,\ldots,k-1$ choose $z_j  = (e^{2\pi im\gamma}\omega_j)_{m\in\mathbb{Z}}\in Z$ such that $d(T^jx,z) \leq 2\varepsilon(Z,n)$, and define also $z_k=z_0$ and $\omega_k=\omega_0$. For $0 \leq j <k$ we have
\[|e^{2\pi i \gamma}\omega_j - \omega_{j+1}| \leq d(Tz_j,z_{j+i}) \leq d(Tz_j,T^{j+1}x)+d(T^{j+1}x,z_{j+1}) \leq 6\varepsilon(Z,n),\]
and it follows that
\[|e^{2\pi i k\alpha}\omega_0-\omega_0| \leq \sum_{j=0}^{k-1}\left|e^{2\pi ij\gamma}\omega_j - e^{2\pi i(j+1)\gamma}\omega_{j+1}\right| \leq 6k\varepsilon(Z,n).\]
However, it is well-known \cite{HW} that there exists $\delta>0$ such that $|e^{2\pi i m\alpha}-1| \geq \delta /m$ for every $m \in \mathbb{N}$, and we deduce that $\varepsilon(Z,n) \geq \delta/6k^2 \geq \delta/6n^2$. 

We conclude that if $\mathsf{A} \subset \Mat_d(\mathbb{C})$ is some compact set of matrices which is isometric to $S^1$, then there exists a minimal invariant set $Z \subset \mathsf{A}^{\mathbb{Z}}$ such that $\varepsilon(Z,n)$ is not $o(n^{-2})$. In particular, the method of Theorem \ref{QBWF} 
is in this case not strong enough even to show that
\[\left|\varrho(\mathsf{A}) - \max_{1 \leq k \leq n}\varrho^-_k(\mathsf{A})\right| = O\left(\frac{1}{n^{2\alpha}}\right),\]
where $\alpha$ is the H\"older exponent of the function $P$ given by Theorem \ref{Tech}. Since $\alpha>0$ is not explicitly known this estimate would anyway be inferior to the estimate of J. Bochi described in the introduction. If we wish to achieve further progress using the methods of the present article, therefore, the key step must be to show that for a given set $\mathsf{A} \subset \Mat_d(\mathbb{C})$ there is an extremal norm $\vvv \cdot \vvv$ for which the set
\begin{equation}\label{extraj}Y = \left\{x \in \mathsf{A}^{\mathbb{Z}} \colon \varrho(\mathsf{A})^{-n}\vvv \mathcal{A}(x,n)\vvv = 1\text{ }\forall\text{ }n \in \mathbb{N}\right\}\end{equation}
contains a minimal set $Z$ such that the quantity $\varepsilon(Z,n)$ decreases with some specified rapidity as a function of $n$.

It should be remarked that the explicit structure of the set $Y$ defined in \eqref{extraj} is for the most part unknown, and so the range of minimal sets $Z$ which may be contained in such a set $Y$ could in principle be quite limited, potentially leading to improved estimates in Theorem \ref{QBWF}. Indeed, the the \emph{finiteness conjecture} of J. Lagarias and Y. Wang, proposed in \cite{LW}, was equivalent to the statement that $Y$ must always contain a periodic orbit. The existence of counterexamples to the finiteness conjecture was established by T. Bousch and J. Mairesse \cite{BM}, with a simpler argument subsequently being given in \cite{BTV}. At present, the only well-understood examples of sets $\mathsf{A}$ in which $Y$ does not contain a periodic orbit have the property that the orbits in $Y$ are ``Sturmian'' or ``balanced'' \cite{BM}. When $Z$ consists of Sturmian orbits one may show that $\varepsilon(Z,n)$ decreases exponentially as a function of $n$, and in particular the arguments used in this article could be applied to obtain an exponential estimate in Theorem \ref{QBWF} in this special case.

\section{Appendix: Proof of the semi-uniform subadditive ergodic theorem}
The proof given below is a condensed exposition of \cite{StSt}, though the hypotheses are slightly weaker and the conclusion slightly stronger. Lemma \ref{A2} below is a mildly strengthened version of \cite[Theorem 1.9]{StSt}; that result generalises a lemma of M. Herman \cite[p.487]{H}, which in turn generalises a well-known theorem of Oxtoby \cite{Ox}. 
\begin{lemma}\label{A2}
Let $T \colon X \to X$ be a continuous map of a compact metric space, and let $f \colon X \to \mathbb{R}\cup\{-\infty\}$ be upper semi-continuous. Then
\[\lim_{n \to \infty} \sup_{x \in X}\frac{1}{n} \sum_{k=0}^{n-1}f(T^kx) = \sup_{\mu \in \mathcal{M}_T} \int f\,d\mu.\]
\end{lemma}
\begin{proof}
It is easy to show that the former quantity is an upper bound for the latter. To show the reverse direction, suppose that $(x_n)_{n=1}^\infty$ satisfies $(1/n)\sum_{k=0}^{n-1}f(T^kx_n) \geq \lambda$ for infinitely many $n \in \mathbb{N}$. Then using Lemma \ref{A1} and the compactness of $\mathcal{M}$ we may choose a weak-* limit point $\mu$ of the sequence of measures $(\mu_n)$ given by $\mu_n = (1/n) \sum_{k=0}^{n-1} \delta_{T^kx_n}$ having the property that $\int f\,d\mu \geq \lambda$. Since clearly $|\int g\,d\mu_n - \int (g \circ T)\,d\mu_n| \to 0$ for every continuous $g$ we have $\mu \in \mathcal{M}_T$.
\end{proof}
\begin{lemma}\label{A3}
Let $Z$ be a compact topological space, and let $(g_n)_{n=1}^\infty$ be a sequence of upper semi-continuous functions from $Z$ into $\mathbb{R}\cup\{-\infty\}$ such that $(g_n(x))_{n=1}^\infty$ is subadditive for every $x \in Z$. Then
\begin{equation}\label{fco}\lim_{n \to \infty} \sup_{z \in Z} \frac{1}{n}g_n(z) = \sup_{z \in Z}\lim_{n \to \infty} \frac{1}{n} g_n(z).\end{equation}
\end{lemma}
\begin{proof}
Let $\lambda>\sup_{z \in Z} \lim_{n \to \infty}(1/n)g_n(z)$. For each $z \in Z$ there exists $n_z>0$ such that $(1/n_z)g_{n_z}(z)<\lambda$, and by upper semi-continuity there is an open neighbourhood $U_z$ of $z$ such that $(1/n_z)g_{n_z}(y)<\lambda$ for all $y\in U_z$. Clearly $\{U_z\colon z \in Z\}$ is an open cover of $Z$ and so we may passing to a finite subcover to deduce that there exist open sets $U_1,\ldots,U_d$ covering $Z$ and integers $n_1,\ldots,n_d$ such that if $z \in U_i$ then $(1/n_i)g_{n_i}(z)<\lambda$. Now take $\hat n = \prod_{i=1}^d n_i$ and for convenience define $m_i = \hat n / n_i \in \mathbb{N}$. If $z \in Z$, then choosing $i$ such that $z \in U_i$ we obtain
\[\frac{1}{\hat n}g_{\hat n}(z) \leq \frac{1}{m_i}\sum_{k=0}^{m_i-1} \frac{1}{n_i}f_{n_i}(z) < \lambda\]
whence $\sup_{z \in Z} (1/\hat{n}) g_{\hat{n}}(z) <\lambda$. Using subadditivity we deduce
\[\lim_{n\to \infty}\sup_{z \in Z} \frac{1}{n}g_n(z) = \inf_{n \geq 1} \sup_{z \in Z}\frac{1}{n}g_n(z) <\lambda,\]
and taking the infimum over $\lambda$ gives one direction of inequality in \eqref{fco}. The reverse inequality is straightforward: for any $y \in Z$ it is clear that
\[\lim_{n \to \infty} \frac{1}{n}g_n(y) \leq \lim_{n \to \infty}\sup_{z \in Z}\frac{1}{n}g_n(z),\]
and taking the supremum over $y$ yields the required result.\end{proof}
\emph{Proof of Theorem \ref{SUSAET}.}
Choose any real number $\lambda > \sup_\mu \lim_n (1/n)\int f_n\,d\mu$. For each $n \in \mathbb{N}$ define a function $g_n \colon \mathcal{M}_T \to \mathbb{R}\cup\{-\infty\}$ by $g_n(\mu) = \int f_n\,d\mu$. By Lemma \ref{A1} this function is upper semi-continuous, and clearly $(g_n(\mu))_{n=1}^\infty$ is subadditive for every $\mu$. By Lemma \ref{A3} we obtain
\begin{equation}\label{ADE1}
 \lim_{n \to \infty}\sup_{\mu\in \mathcal{M}_T} \frac{1}{n}\int f_n\,d\mu
 = \sup_{\mu \in \mathcal{M}_T} \lim_{n \to \infty}\frac{1}{n}\int f_n\,d\mu
\end{equation}
and it follows that there exists $n_1>0$ such that $(1/n_1)\int f_{n_1}\,d\mu <\lambda$ for all $\mu \in \mathcal{M}_T$. Let $M=\sup f_1$. Applying Lemma \ref{A2} to $f_{n_1}$ it follows that for all sufficiently large integers $n_2$ we have uniformly for each $x \in X$
\begin{align*}n_1 f_{n_1n_2}(x) &\leq \sum_{r=0}^{n_1-1}\left(f_r(x) + f_{n_1-r}\left(T^{n_1(n_2-1)+r}x\right)+\sum_{q=0}^{n_2-2} f_{n_1}\left(T^{qn_1+r}x\right) \right)\\
&\leq Mn_1^2 + \sum_{k=0}^{n_1(n_2-1)-1}f_{n_1}\left(T^kx\right) < Mn_1^2 + n_1^2(n_2-1)\lambda\end{align*}
where we have used the notation $f_0 \equiv 0$ to simplify the presentation. Hence,
\[\inf_{r \geq 1} \sup_{x \in X}\frac{1}{r} f_r(x) \leq \lim_{n_2 \to \infty} \frac{Mn_1^2+n_1^2(n_2-1)\lambda}{n_1^2n_2} = \lambda.\]
Taking the infimum over $\lambda$ and using subadditivity we obtain
\begin{equation}\label{ADE2}\lim_{n \to \infty}\sup_{x \in X}\frac{1}{n}f_n(x) = \inf_{n \geq 1}\sup_{x \in X}\frac{1}{n}f_n(x) \leq \sup_{\mu \in \mathcal{M}_T} \lim_{n \to \infty}\frac{1}{n}\int f_n\,d\mu.\end{equation}
By Lemma \ref{nonemptiness} there exists an ergodic measure $\nu$ which attains this last supremum. Applying the subadditive ergodic theorem it follows that
\begin{equation}\label{ADE3}\sup_{\mu \in \mathcal{M}_T} \lim_{n \to \infty}\frac{1}{n}\int f_n\,d\mu \leq\sup_{x \in X} \limsup_{n \to \infty} \frac{1}{n}f_n(x).\end{equation}
Since for every $z \in X$ we clearly have
\[\limsup_{n \to \infty} \frac{1}{n}f_n(z) \leq \lim_{n \to \infty} \sup_{x \in X}\frac{1}{n}f_n(x)\]
we deduce
\begin{equation}\label{ADE4}\sup_{x \in X} \limsup_{n \to \infty} \frac{1}{n}f_n(x) \leq \lim_{n \to \infty} \sup_{x\in X}\frac{1}{n}f_n(x).\end{equation}
Combining \eqref{ADE1}, \eqref{ADE2}, \eqref{ADE3} and \eqref{ADE4} serves to complete the proof.
\section{Acknowledgments}
This research was supported by EPSRC grant EP/E020801/1. The author would like to thank M. Pollicott for suggesting the reference \cite{FLQ}, and J. Hirsch for pointing out an error in an earlier version of this paper.
\bibliographystyle{amsplain}
\bibliography{QBWF}
\end{document}